\documentclass[a4paper, 10pt, twoside, notitlepage]{amsart}

\usepackage[utf8]{inputenc}
\usepackage{color}
\usepackage{amsmath} 
\usepackage{amssymb} 
\usepackage{amsthm}
\usepackage{geometry}
\usepackage{graphicx}
\usepackage{esint}
\usepackage[colorlinks=true,linkcolor=blue]{hyperref}

\theoremstyle{plain}

\newtheorem{prop}{Proposition}[section]
\newtheorem{lem}[prop]{Lemma}

\newtheorem{defi}[prop]{Definition}
\newtheorem{rmk}[prop]{Remark}

\newtheorem{con}[prop]{Convention}

\newcommand {\R} {\mathbb{R}} 
 \newcommand {\N} {\mathbb{N}}
 
\newcommand {\p} {\partial}

\newcommand {\D} {\Delta}

\newcommand {\sgn} {\text{sgn}}
\newcommand {\supp} {\text{supp}}

\DeclareMathOperator{\G}{\mathcal{G}}
\DeclareMathOperator{\Je}{\mathcal{J}_{\epsilon}}
\DeclareMathOperator{\Jed}{\mathcal{J}_{\epsilon,\delta}}

\DeclareMathOperator {\dist} {dist}

\DeclareMathOperator {\sign} {sgn}

\DeclareMathOperator{\F} {\mathcal{F}}

\author{Angkana R\"uland}

\address{
Mathematical Institute of the University of Oxford, Andrew Wiles Building, Radcliffe Observatory Quarter, Woodstock Road, OX2 6GG Oxford, United Kingdom }
\email{ruland@maths.ox.ac.uk}

\begin{document}

\title[Quantitative Estimates for Truncated Riesz Transforms]{Quantitative Invertibility and Approximation for the Truncated Hilbert and Riesz Transforms}

\begin{abstract}
In this article we derive quantitative uniqueness and approximation properties for (perturbations) of Riesz transforms. Seeking to provide robust arguments, we adopt a PDE point of view and realize our operators as harmonic extensions, which makes the problem accessible to PDE tools. In this context we then invoke quantitative propagation of smallness estimates in combination with qualitative Runge approximation results. These results can be viewed as quantifications of the approximation properties which have recently gained prominence in the context of nonlocal operators, c.f. \cite{DSV14}, \cite{DSV16}.
\end{abstract}



\thanks{
The author gratefully acknowledges a Junior Research Fellowship at Christ Church. She would like to thank Mikko Salo for helpful discussions related to this project.}

\maketitle

\section{Introduction}
In this note we derive \emph{quantitative unique continuation} and \emph{approximation} results for the truncated Hilbert transform, truncated Riesz transforms and for certain classes of perturbations of these. 
These operators should be regarded as some of the simplest possible model problems, which display typical features of elliptic nonlocal operators of fractional Laplacian type (c.f. \cite{DSV14}, \cite{DSV16}, \cite{Rue15}) in that:
\begin{itemize}
\item They satisfy very strong uniqueness properties (c.f. Lemma \ref{lem:Hilbert} (a)).
\item They allow for very good approximation properties (c.f. Lemma \ref{lem:Hilbert} (b)).
\end{itemize}
As pointed out in \cite{GSU16} these two properties are dual with respect to each other (c.f. also \cite{L56}, \cite{B62}, \cite{B62a} for similar Runge type approximation and duality results in the context of local equations). The main objective of this note is to provide \emph{quantitative} versions of both of these properties by means of robust PDE tools in the situation of the described model problems. We expect that this point of view can be generalized to a much broader class of nonlocal problems. \\

Let us begin by discussing the one-dimensional situation: Here we study (modifications of) the Hilbert transform
\begin{align*}
Hf(x) := p.v. \int\limits_{\R} \frac{f(t)}{x-t}dt.
\end{align*}
The Hilbert transform is a prototypical singular integral operator (defined through a principal value integral), which arises in many different applications (c.f. \cite{G08}). A related operator, the \emph{truncated
Hilbert transform} plays an important role in medical imaging and has thus attracted a substantial amount of attention (c.f. \cite{DNCK06} and the references therein). Considering two open bounded intervals $I,J \subset \R$, it is defined as
\begin{align*}
H_{I,J}f := \chi_J H(f \chi_{I}) \mbox{ for all } f \in L^2_{loc}(\R),
\end{align*}
where $\chi_I, \chi_J$ denote the characteristic functions of the two intervals $I$ and $J$. In the sequel, we will also simply use the notation $H_I f$ to refer to the function $H(f \chi_{I})\in L^2(\R)$ for $f \in L^2(I)$.
In the context of medical imaging, the following question is of interest:
\begin{itemize}
\item[Q:]\emph{Is it possible to recover $f\chi_I$ from the knowledge of $H_{I,J}f$?}
\end{itemize}
Mathematically this can be translated into an investigation of the mapping properties of the operators $H_{I,J}$ (in their particular injectivity properties).
These depend crucially on the relative location of the intervals $I,J$, c.f. \cite{APS13}:
\begin{itemize}
\item[(a)] In contrast to the whole space situation, the operators  $H_{I,I}$ are no longer continuously invertible in $L^2(I)$ for bounded intervals $I\subset \R$. However, continuous inversion can be recovered in suitably weighted spaces, c.f. \cite{APS96}, \cite{T51}. Thus, given $H_{I,I} f$, the reconstruction of $f \chi_I$ is stable in suitable function spaces.
\item[(b)] If $\overline{I}\cap \overline{J}= \emptyset$, the operator $H_{I,J}$ is compact. Thus, by basic functional analysis, it no longer has a bounded inverse.
The associated inverse reconstruction problem is consequently strongly ill-posed in general, c.f. \cite{AK14}, \cite{ADK15}. However, in \cite{APS13} Alaifari, Pierce and Steinerberger observed that if certain a priori information is given, it becomes possible to ``continuously" invert the problem (c.f. also the general philosophy outlined in \cite{T43}, \cite{J60}, \cite{B89}, and \cite{LS16} for a similar application of this strategy to the truncated Fourier and Laplace transforms). This can further be precisely quantified.
\item[(c)] If $\overline{I}\cap \overline{J} \neq \emptyset$, but $I \neq J$, one can hope for improved stability properties of the recovery problem. Again this however depends on the precise relation of the intervals $I,J$. For instance in the \emph{interior} situation, where $\overline{I}\subset J$, it was shown in \cite{ADK16} that it is possible to establish a Hölder continuous dependence on the data (c.f. also the comparison of local versus global stability estimates in \cite{ARV09}). As expected, this however degenerates as $I$ approaches $J$.
\end{itemize}
In the sequel, we focus on the worst case scenario (b), and provide quantitative invertibility estimates in this setting. To this end, we view the invertibility problem in a PDE framework, which allows us to borrow tools from this context.

\subsection{Main results and ideas}
As the main objective of this article, we seek to provide quantitative uniqueness and approximation properties for operators, which are similar to the Hilbert transform, by means of robust PDE tools. Here we only appeal to propagation of smallness estimates (which can be viewed as consequences of associated Carleman estimates) and variational principles. This provides an alternative approach to methods, which are used in the literature (c.f. \cite{LP61}, \cite{APS13} and the references therein). In particular, we do not use an explicit characterization of the singular values of the Hilbert transform. As a consequence, it is possible to generalize our results from the one dimensional to the higher dimensional setting (including variable coefficients), c.f. Sections \ref{sec:nD1}, \ref{sec:Perturb}. 

\subsubsection{Quantitative almost invertibility}
Motivated by the examples of medical applications, we are interested in conditional stability estimates for the truncated Hilbert transform. As in \cite{APS13} we identify oscillations as the ``only" obstruction for reconstructing a function $f\in L^2(I)$ from its truncated Hilbert transform $H_{I,J}f\in L^2(J)$.

\begin{prop}[Quantitative unique continuation]
\label{prop:HT}
Let $I,J$ be open, bounded intervals such that $\overline{I}\cap \overline{J} = \emptyset$.
Denote by $H_J:L^2(\R) \rightarrow L^2(\R)$ the truncated Hilbert transform with respect to the open interval $J$ and assume that 
 $g\in H^1(J)$.
 Then, there exist constants  $\tilde{\sigma} >0$ and $C>1$ (which depend only on the relative size and distance of the intervals $I$ and $J$ (c.f. Remark \ref{rmk:dependence}) such that
 \begin{align}
 \label{eq:almost_invert1}
  \| g\|_{H^1(J)}
  \leq  \exp\left( C\left(1+\frac{\|g\|_{H^1(J)}^{\tilde{\sigma}} }{\| g\|_{L^2(J)}^{\tilde{\sigma}}} \right)\right) \|H_J g\|_{L^2(I)}.
 \end{align} 
\end{prop}

Let us briefly comment on this result: As the explicit characterization of the singular values of the Hilbert transform shows, the
\begin{itemize}
\item the exponential nature of the estimate \eqref{eq:almost_invert1} is optimal, 
\item the exponent $\tilde{\sigma}$ which appears in the estimate is in general far from the optimal one, which can be obtained from the singular value characterization for the truncated Hilbert transform, c.f. \cite{APS13}. 
\end{itemize}
However, in contrast to the methods which rely on explicit eigenfunction asymptotics, the present approach is very robust and generalizes to the higher dimensional situation with quite general domain geometries. \\
We further remark that \eqref{eq:almost_invert1} can also be read as a generalized injectivity or quantitative unique continuation result for the truncated Hilbert transform. 

\subsubsection{Quantitative approximation properties} 
While it is well-known that for disjoint open intervals $I,J$ the truncated Hilbert transform has a dense image as a mapping from $L^2(I)$ to $L^2(J)$ (c.f. Corollary \ref{cor:inj_dens} below, in which this is derived by relying on the ideas from \cite{GSU16}), its more precise approximation properties are only quantified in particular situations (c.f. \cite{ADK16}). Relying on propagation of smallness results and variational principles, we provide explicit bounds on the \emph{cost} of approximation: As the inversion operator is not continuous, it is expected that an increase of the approximation quality leads to an increase in the norm (whose size we interpret at the \emph{cost} of approximation) of the approximating functions. We provide upper bounds on this cost.

\begin{prop}
\label{prop:HT2}
Let $I,J$ be open, bounded intervals such that $\overline{I}\cap \overline{J} = \emptyset$.
Let $\epsilon>0$ and let $h\in L^2(J)$. Then, there exist constants  $\tilde{\sigma} >0$ and $C>1$ (which depend only on the relative size and distance of the intervals $I$ and $J$ (c.f. Remark \ref{rmk:dependence})) and functions $f\in L^2(I)$ such that
\begin{align}
\label{eq:approx}
\|h- H_I f\|_{L^2(J)} \leq \epsilon  \mbox{ and }
\|f\|_{L^2(I)} \leq e^{C(1+\|h\|_{H^1(J)}^{\sigma}/\epsilon^{\sigma})} \|h\|_{L^2(J)}.
\end{align}
\end{prop}

In Proposition \ref{prop:HT}, the exponential dependence on $\epsilon$ is optimal, while the explicit power $\sigma$ is certainly non-optimal. This result should be viewed in the context of the recent interest in approximation properties of nonlocal operators, c.f. \cite{DSV14}, \cite{DSV16}, \cite{GSU16}. Here we do not only provide an approximation result, but also give \emph{quantitative estimates} on the \emph{cost of control}. In the context of control problems the question on the cost of control has a long tradition (c.f. for example \cite{LR95}, \cite{FZ00} and the references therein). For the Hilbert transform these bounds can also be obtained by means of the singular value characterization, but for more general nonlocal operators, for instance including the ones which are treated in Sections \ref{sec:nD1}, \ref{sec:Perturb}, these bounds seem to be new.

\subsubsection{Main ideas}
Quantitative propagation of smallness estimates lie at the heart of both the uniqueness and the approximation results from Proposition \ref{prop:HT} and \ref{prop:HT2}. These are obtained with the aid of robust PDE methods in the form of three balls arguments (which themselves are based on Carleman inequalities), interpolation estimates and elliptic estimates. Similarly to the unique continuation results for nonlocal elliptic operator, the problem at hand becomes accessible to these methods after realizing it by means of a harmonic extension as an elliptic \emph{local} operator in the upper half-plane (c.f. Lemma \ref{lem:Hilbert}).

\subsection{Generalizations to higher dimensions and perturbations}

A key feature of our method is its robustness. In contrast to the more precise singular value decompositions, it depends much less sensitively on the specific geometry of the underlying domains. Thus, the results of Propositions \ref{prop:HT} and \ref{prop:HT2} directly generalize to the context of Riesz transforms (c.f. Propositions \ref{prop:approx_Riesz}, \ref{prop:quant_unique2_Riesz}). Moreover, it is possible to treat (a restricted class of) variable coefficient perturbations of these operators (c.f. Propositions \ref{prop:approx_Riesz_a}, \ref{prop:quant_unique2_Riesz_a}). We however expect that similar methods can be used for a much larger class of operators.

\subsection{Organization of the article}
We conclude the introduction by commenting on the organization of the remaining article. In Section \ref{sec:QHT} we provide detailed arguments for the quantitative uniqueness and approximation properties of the Hilbert transform. The core of the argument consists of the propagation of smallness estimates from Section \ref{sec:small}. With the arguments for the Hilbert transform at hand, in Sections \ref{sec:nD1} and \ref{sec:Perturb} we further explain generalizations of these ideas to Riesz transforms and perturbations of these. 

\section[Quantitative Estimates]{Quantitative Estimates for the Truncated Hilbert Transform}
\label{sec:QHT}

As an introduction to the ideas leading to the quantitative properties of the nonlocal operators at hand, we present detailed arguments for the Hilbert transform. While this problem is also accessible by other arguments, e.g. by asymptotic properties of the associated singular value decomposition which yield very precise estimates, we view the Hilbert transform as a model problem, which allows us to introduce our robust PDE based tools. With these arguments at hand, extensions to more general operators and the higher dimensional situation are then straightforward. These are explained  in the later parts of the article, c.f. Sections \ref{sec:nD1} and \ref{sec:Perturb}. \\
The section is divided into four main parts: We first briefly recall the extension point of view for the Hilbert transform in Section \ref{sec:HT}. Then, based on this, we deduce first qualitative properties of the Hilbert transform in Section \ref{sec:qual}. With this at hand, in Section \ref{sec:small} we proceed to the key ingredient of our argument and discuss quantitative propagation of smallness estimates. As direct consequences of this, we finally infer the quantitative injectivity and approximation properties of Propositions \ref{prop:HT} and \ref{prop:HT2}.

\subsection{Realizations of the truncated Hilbert transform}
\label{sec:HT}

In this section we recall the characterization of the Hilbert transform in terms of a harmonic extension operator (c.f. for instance Section 4.1.2. in \cite{G08} on the connection of the Hilbert transform and analytic functions). This harmonic extension point of view makes the problem accessible to robust PDE techniques, which we will exploit systematically in the following sections.

\begin{lem}[\cite{G08}, Theorem 4.1.5]
\label{lem:Hilbert}
Let $H:L^2(\R) \rightarrow L^2(\R)$ denote the Hilbert transform. Then it is also possible 
 to realize this operator by means of the Neumann harmonic extension as 
 \begin{align*}
  H f (x_1)= \p_1 N (f)(x_1,0),
 \end{align*}
where $N:C^{\infty}_0(\R)\rightarrow H^1_{loc}(\R^2)$, $f\mapsto u:=G_N \ast_{x_1} f$ with $G_N(x) = \ln(|x|)$, denotes a solution 
operator to
\begin{align*}
-\D u& = 0 \mbox{ in } \R^{2}_+,\\
\p_2 u & = f \mbox{ on } \R \times \{0\}.
\end{align*}
\end{lem}

\begin{rmk}
\label{rmk:interpret}
Here and in the sequel, the Hilbert transform is understood in the principal value sense, i.e.
\begin{align*}
Hf(x) := p.v. \int\limits_{\R} \frac{f(t)}{x-t}dt 
:= \lim\limits_{\epsilon \rightarrow 0} \int\limits_{|x-t|>\epsilon} \frac{f(t)}{x-t}dt. 
\end{align*}
By cancellation, this is well-defined on Lipschitz functions $f$ and can then be extended by continuity to $f\in L^p(\R^n)$ with $p\in(1,\infty)$ (c.f. the discussion in \cite{G08}).\\
The operator $\p_1 N (f)(x_1,0)$ is regarded as an $L^p$ with $1\leq p <\infty$ limit of the conjugate Poisson kernel:
\begin{align*}
\p_1 N(f)(x_1,0) =  \lim\limits_{z\rightarrow 0} \int\limits_{\R} f(y)\frac{x-y}{|(x-y,z)|^2} dy.
\end{align*}
This limit agrees with the Hilbert transform \cite{G08}, Theorem 4.1.5.
\end{rmk}

\begin{rmk}
\label{rmk:unique}
We remark that modifying the function $u$ from Lemma \ref{lem:Hilbert} by adding constants to it, does not change the properties of the mapping $f\mapsto \p_1 N(f)$. In the sequel, we will frequently exploit this observation, when dealing with local properties of harmonic functions.
\end{rmk}

\begin{proof}
The claimed identity follows by either adopting the principal value or the Fourier approach and by recalling the principal value or Fourier definition of the Hilbert transform. Indeed, as explained in \cite{G08}, Theorem 4.1.5, for any $1\leq p <\infty$ and $f\in L^p(\R)$ it holds that
\begin{align*}
\p_1 N(f)(x_1, x_2) - H^{x_2}(f) \rightarrow 0 \mbox{ in } L^p
\end{align*}
as $x_2 \rightarrow 0$, where
\begin{align*}
H^{\epsilon}(f) = \int\limits_{|x-y|>\epsilon} \frac{f(t)}{x-t}dt,
\end{align*}
denotes the standard regularization of the (whole space) Hilbert transform.
\end{proof}

\begin{rmk}
Alternatively, assuming that all quantities are well-defined (which is e.g. the case if $\F f(0)=0$), on the Fourier side we obtain that (up to constants) $\F(G_N \ast_{x_1} f)(\xi) = \frac{\F(f)(\xi)}{|\xi|}e^{-x_{2} |\xi|} $. Hence, $ \F( \p_1 G_N \ast_{x_1} f)(\xi) = i\sign(\xi) \F(f)(\xi)e^{-x_{2} |\xi|} $.
\end{rmk}

As an important, well-known property of the truncated Hilbert transform, we note that $H_{I,J}$ and $-H_{J,I}$ are adjoint operators:

\begin{lem} 
\label{lem:adj}
Let $I,J \subset \R$ be two intervals. Then we have that
\begin{align*}
(H_I f,g)_{L^2(J)} = -(H_J g,f)_{L^2(I)}.
\end{align*}
\end{lem}

\begin{proof}
The proof is a direct consequence of the multiplier characterization of the Hilbert transform and Plancherel's theorem:
\begin{align*}
(H_I f, g)_{L^2(J)} 
&= (H (\chi_I f), \chi_J g)_{L^2(\R)} 
= \left(i \sgn(\xi) \F(\chi_I f), \F(\chi_J g)\right)_{L^2(\R)}\\
&= -\left( \F(\chi_I f), i \sgn(\xi) \F(\chi_J g)\right)_{L^2(\R)}
= -(\chi_I f, H(\chi_J g))_{L^2(\R)} \\
&= -(f, H_J g)_{L^2(I)}. \qedhere
\end{align*}
\end{proof}

\begin{rmk}
\label{rmk:alt_proof}
As an alternative to the arguments by means of the multiplier characterization, it would also have been possible to use the harmonic extension characterization in connection with an integration by parts to prove Lemma \ref{lem:adj}.
\end{rmk}

\subsection[Qualitative unique continuation and approximation results]{Qualitative unique continuation and approximation results for the truncated Hilbert transform}
\label{sec:qual}

As a first consequence of the localization and extension point of view from Lemma \ref{lem:Hilbert}, in this section we exploit (weak) boundary unique continuation results for the Laplacian \cite{EA97}, \cite{AEK95}, \cite{KN98} (or equivalently, weak unique continuation results for the half-Laplacian, c.f. \cite{Rue15}). This entails two well-known properties of the truncated Hilbert transform on $L^2$: We deduce its injectivity and the fact that it has a dense image.

\begin{lem}[Injectivity and density]
\label{cor:inj_dens}
Let $I,J \subset \R$ be open intervals with $\overline{I}\cap \overline{J} = \emptyset$.
Denote by $H_I$ the truncated Hilbert transform associated with the interval $I$. Then, we have the following properties:
\begin{itemize}
\item[(a)] If for some $f\in L^2(I)$ and for all $x \in J$ it holds that $H_I f (x) = 0$, then $f=0$ as a function in $L^2(I)$.
\item[(b)] The set $\{H_I g: g \in C^{\infty}_0(I)\}$ is dense in $L^2(J)$.
\end{itemize}
\end{lem}

\begin{proof}
The statement of (a) follows from a combination of Lemma \ref{lem:Hilbert} and weak boundary unique continuation results for the Laplacian (or equivalently, weak unique continuation results for the half-Laplacian). Indeed, by Lemma \ref{lem:Hilbert} the associated Neumann harmonic extension $N(f)$ satisfies
\begin{align*}
\D (N(f)) &= 0 \mbox{ in } \R^2_+,\\
\p_2 (N(f)) & = 0 \mbox{ on } J \times \{0\},\\
N(f) & = c \mbox{ on } J \times \{0\},
\end{align*}
for some constant $c\in \R$.
Consequently the function $\tilde{f}:= N(f) - c$ solves
\begin{align*}
\D \tilde{f} &= 0 \mbox{ in } \R^2_+,\\
\p_2 \tilde{f} & = 0 \mbox{ on } J \times \{0\},\\
\tilde{f} & = 0  \mbox{ on } J \times \{0\}.
\end{align*}
By (boundary) weak unique continuation this however implies that $\tilde{f}=0$, from which we infer that $N(f) = c$ in $\R^2_+$. Since on $I \times \{0\}$ we consequently obtain that $\p_2 N(f)=0$, we have that
\begin{align*}
0= \p_2 N(f) =f \mbox{ on } I \times \{0\}.
\end{align*}
This entails that $c=f=0$.\\
The density property (b) follows by duality and a reduction to the unique continuation property (a). Indeed, by Hahn-Banach it suffices to show that if there existed $v\in L^2(J)$ with 
\begin{align}
\label{eq:HB}
(v,H_I g)_{L^2(J)}=0 \mbox{ for all } g\in C^{\infty}_0(I),
\end{align}
then $v=0$. By the characterization of the Hilbert space adjoint (c.f. Lemma \ref{lem:adj}), \eqref{eq:HB} 
however implies that 
\begin{align*}
( H_J v, g)_{L^2(I)} =(v,  H_I g)_{L^2(J)}= 0 \mbox{ for all } g \in 
C^{\infty}_0(I), 
\end{align*}
which in turn yields that $(H_J v)(x) = 0$ for all $x\in I$. By the result from (a) this allows us to conclude that $v=0$ as a function in $L^2(J)$. 
\end{proof}

\subsection{Propagation of smallness}
\label{sec:small}

In this section we derive a crucial propagation of smallness estimate, which forms the core of our argument. Technically, it is essentially a consequence of the ideas from \cite{ARV09} and relies on a combination of
\begin{itemize}
\item a three balls lemma,
\item elliptic regularity estimates,
\item appropriate trace theorems.
\end{itemize}
Before presenting this result, we introduce some notation and conventions related to the intervals $I,J$, in order to normalize the set-up. 

\begin{con}
\label{con:int}
Given two open intervals $I,J$ with $\overline{I}\cap \overline{J}=\emptyset$, we may without loss of generality assume that $J=(0,1)$. This follows from translation and rescaling. Moreover, we may suppose that $I=(a,b)$ for some $-\infty<a<b<0$. We abbreviate the remaining two parameters related to the intervals $I,J$ by
\begin{align*}
d_I:= \dist(I,J)=|b|, \ l_I:= b-a.
\end{align*}
For convenience of notation, we further define $h_1:= \min\{d_I/4,1/4\}$ and for an arbitrary open interval $\tilde{I}$ and $h>0$ we set
\begin{align*}
 \tilde{I}_h:= \{x\in \R: \dist(\tilde{I},x)<h\}.
\end{align*}
\end{con}

In the sequel, we will always assume that the setting has been normalized to the situation in Convention \ref{con:int}. Using this assumption, we formulate our main propagation of smallness result:

\begin{prop}
\label{prop:small}
Let $I,J$ be as in Convention \ref{con:int}.
Assume that $g \in L^2(J)$ and consider the function
 $u:=N(g):\R^2_+ \rightarrow \R$ which solves
\begin{align*}
\D u &= 0 \mbox{ in } \R^2_+,\\
\p_2 u & = \chi_J g \mbox{ on } \R \times \{0\}.
\end{align*}
Then there exist constants $\sigma >0$, $C>1$ (depending on the relative distance and the relative length $d_I,l_I$ of the two intervals $I,J$) such that for any $\epsilon>0$, $\delta \in (0,1)$ we have
\begin{equation}
\label{eq:quant_control}
\begin{split}
\|\p_2 u\|_{L^2(J \times \{\delta\})} 
&\leq \exp(C(|\ln(\epsilon)|+1)/\delta^{\sigma}) \|\p_1 u\|_{L^2(I \times \{0\})} + \frac{\epsilon}{2} \|g\|_{L^2(J)}\\
&= \exp(C(|\ln(\epsilon)|+2)/\delta^{\sigma}) \|H_J g\|_{L^2(I)} + \frac{\epsilon}{2} \|g\|_{L^2(J)}.
\end{split}
\end{equation}
\end{prop}

\begin{proof}
The proof relies on an elliptic propagation of smallness result (which is similar to the arguments from \cite{ARV09} but uses the interpolation estimate of Lebeau and Robbiano \cite{LR95}, in order to propagate the Neumann data). We argue in three steps:\\

\emph{Step 1: A three balls estimate.}
We recall the following interior three balls inequality for harmonic functions
\begin{align}
\label{eq:3sphere}
\|u\|_{L^2(B_{r}^+(x_0))}
 \leq C \|u\|_{L^2(B_{r/2}^+(x_0))}^{\gamma} \|u\|_{L^2(B_{2r}^+(x_0))}^{1-\gamma}.
\end{align}
In our set-up this holds for all points $x_0 \in \R^2_+$ and for all radii $r>0$ with the property that $\dist(x_0, J) \geq 4r$; $\gamma\in(0,1)$ is a universal constant.
We emphasize that in the setting of Proposition \ref{prop:small}, we only have to require a control on the distance to $J$ (and not to the whole set $\R \times \{0\}$), since in $(\R \setminus J)\times \{0\}$ the solution can be extended as a harmonic function into the lower half-plane by an even reflection. This permits us to use interior estimates in these regions as well.\\
Next, we consider a chain of balls, $\bigcup\limits_{j=1}^{N} B_{r_j}(x_j)$, which connects the sets $I \times [1/2,1]$ and $J\times [\delta/2,2\delta]$. We remark that it is possible to choose $N \sim C(d_I,l_J)(|\log(\delta)|+1)$.
Applying the estimate \eqref{eq:3sphere} to the function $v=\p_2 u$ (which for any $\delta>0$ is harmonic in the set $\R \times (\delta/2,\infty)$) and iterating the estimate along the chain of balls then results in
\begin{equation}
\label{eq:3balls1}
\|v\|_{L^2(J_{h_1}\times [\delta/2,2\delta])}
 \leq C^N \|v\|_{L^2(I\times [1/2,1])}^{\gamma^N} \|v\|_{L^2(K)}^{1-\gamma^N}.
\end{equation}
Here $K\subset \R^2_+$ denotes a slight fattening of the chain of balls, which has been used to propagate the smallness condition (we can for instance fatten by a factor $\min\{h_1,\delta/100\}$). In particular, it is possible to ensure that $\dist(K,J\times \{0\})\geq \delta/10$. Moreover, we recall that by Convention \ref{con:int} the set $J_{h_1}$ denotes a slight fattening of $J$. \\
By invoking Caccioppoli's estimate and trace bounds, \eqref{eq:3balls1} can further be upgraded to yield
\begin{equation}
\label{eq:3balls2}
\|v\|_{H^1(J_{h_1}\times [\delta/2,2\delta])} + \|v\|_{L^2(J_{h_1}\times \{\delta\})}
 \leq C^N \delta^{-1}\|v\|_{L^2(I\times [1/2,1])}^{\gamma^N} \|v\|_{L^2(K)}^{1-\gamma^N}.
\end{equation}

\emph{Step 2: Propagation of the boundary data.}
We recall the interpolation estimate of Lebeau and Robbiano \cite{LR95}, which allows us to propagate information from the boundary data. In our set-up we use it in the form
\begin{equation}
\label{eq:RL1}
\begin{split}
\|v\|_{L^2(I\times [1/2,1])} 
&\leq C \|u\|_{H^1(I\times [1/2,1])}\\
&\leq C \|u\|_{H^1(I_{h_1}\times [0,2])}^{\alpha}(\|u\|_{L^2(I \times \{0\})} + \|\p_1 u\|_{L^2(I \times \{0\})})^{1-\alpha}.
\end{split}
\end{equation}
Here $I_{h_1}$ is a slight fattening of $I$, which is chosen such that $I_{h_1} \cap J_{h_1} = \emptyset$, and $\alpha \in (0,1)$ is a universal constant.
Since $v$ is defined by differentiating $u$, we may, without loss of generality, subtract a constant from $u$ in such a way that there exists a point $\bar{x}_0 \in I$ with $u(\bar{x}_0)=0$. This permits us to bound the boundary $L^2$ norm on the right hand side of \eqref{eq:RL1} by the fundamental theorem and to infer a bound in terms of the Hilbert transform
\begin{equation}
\label{eq:RL2}
\begin{split}
\|v\|_{L^2(I\times [1/2,1])} 
&\leq C \|u\|_{H^1(I_{h_1}\times [0,2])}^{\alpha}\|H_J g\|_{L^2(I \times \{0\})}^{1-\alpha}.
\end{split}
\end{equation}
Combining this with the estimate from (\ref{eq:3balls2}) thus entails that
\begin{align}
\label{eq:3balls3}
\|v\|_{L^2(J_{h_1}\times \{\delta\})} + \|v\|_{H^1(J_{h_1}\times [\delta/2,2\delta])}
 \leq C^N \delta^{-1} \|H_J g\|_{L^2(I \times \{0\})}^{(1-\alpha)\gamma^N} \|u\|_{H^1(\tilde{K})}^{1-\gamma^N + \alpha \gamma^N}.
\end{align}
Here $\tilde{K}$ is a fattening of the set $K\cup (I_{h_1}\times [0,2])$.\\

\begin{figure}[t]
\centering
\includegraphics[width=0.7 \textwidth]{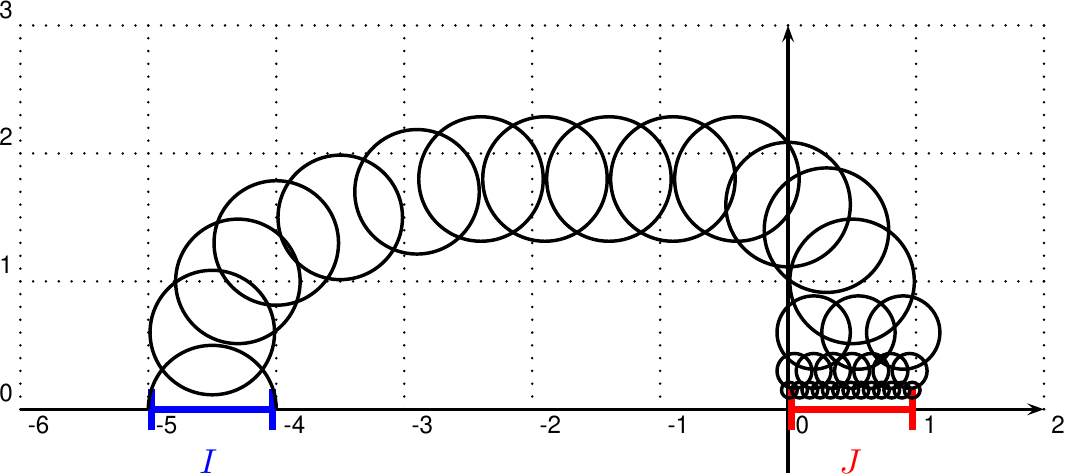}
\caption{Propagation of smallness through a chain of balls. The figure illustrates the chain of balls that is used in transferring information from the interval $I$ to the interval $J$. As we seek to infer information on $\|\p_2 u\|_{L^2(J \times \{0\})}$, we have to approach this part of the boundary from the interior of the upper half-plane. In particular, this enforces that the balls in the chain of balls argument decrease with their distance to $J\times \{0\}$. This accounts for the $\delta$-dependence on the number, $N>1$, of necessary balls, c.f. Remark \ref{rmk:dependence}.}
\end{figure}

\emph{Step 3: Conclusion.} In order to conclude our argument, we observe that 
\begin{align*}
\|u\|_{H^1(\tilde{K})} \leq \frac{C}{\delta}\|g\|_{L^2(J)},
\end{align*}
which follows from noting that
\begin{itemize}
\item the mapping $L^2(J) \ni g \mapsto u \in H^1_{loc}(\R^2_+)$ is smoothing away from $J$,
\item $\dist(\tilde{K},J)\geq \frac{\delta}{100}$,
\item and by using the expression for the fundamental solution in the upper half-plane.
\end{itemize}
Inserting this and the expression for $N \sim C(d_I,l_I)(|\log(\delta)|+1)$ into \eqref{eq:3balls3} and applying Young's inequality leads to
\begin{align*}
\|\p_2 u\|_{L^2(J \times \{\delta\})} 
& \leq C^N \delta^{-2} \|\p_1 u\|_{L^2(I \times \{0\})}^{(1-\alpha)\gamma^N} \|g\|_{L^2(J \times \{0\})}^{1-\gamma^N + \alpha \gamma^N}\\
 & \leq e^{C(l_I,d_I,\alpha)(1+|\ln(\epsilon)|)/ \delta^{\sigma}} \|H_J g\|_{L^2(I)} + 
 \frac{\epsilon}{2} \|g\|_{L^2(J)},
\end{align*}
where $\sigma = C|\ln(\gamma)| $.
This concludes the argument.
\end{proof}

\begin{rmk}[Dependence on $I,J$] 
\label{rmk:dependence}
We emphasize that in the argument from above, the dependence on the intervals $I,J$ enters only through the constant $N$, which is related to the choice of the path that is used in the chain of balls argument. Going through the argument of Proposition \ref{prop:small} carefully, we can provide more quantitative dependences on $d_I, l_I$: It is possible to choose $N$ such that
\begin{align*}
N \sim C (|\log(\delta)|+\max\{1,-\log(l_I)\} + d_I).
\end{align*}
In particular, $C(d_I,l_I)>1$ in \eqref{eq:quant_control} can be roughly estimated by $C(d_I,l_I)\leq \bar{C}^{N \gamma^{-N}}$, 
where $\bar{C}>1 $ is an absolute constant.
\end{rmk}

Next we observe that it is possible to replace the term $\|\p_2 u\|_{L^2(J\times \{\delta\})}$, which appears in the left hand side of \eqref{eq:quant_control}, by $\|\p_2 u\|_{L^2(J\times \{0\})}$ under suitable a priori knowledge (in terms of regularity) on $g$.

\begin{lem}
\label{lem:boundary}
Let $\delta \in (0,1)$ and let $J$ be as in Convention \ref{con:int}.
Assume that $g \in H^1(J)$ and consider the function $u:=N(g):\R^2_+ \rightarrow \R$ which solves 
\begin{align*}
\D u &= 0 \mbox{ in } \R^2_+,\\
\p_2 u & = \chi_J g \mbox{ on } \R \times \{0\}.
\end{align*}
Then there exists a constant $C>1$ such that
\begin{align*}
\|g\|_{L^2(J )} \leq C\left( \|\p_2 u\|_{L^2(J \times \{\delta\})} +\delta^{2/5}  \|g\|_{L^2(J )} + \delta^{2/5} \|\p_1 g\|_{L^2(J)}) \right).
\end{align*}
\end{lem}

\begin{rmk}
\label{rmk:powers}
The exponent $2/5$ is clearly non-optimal. As, due to the chain of balls arguments from Proposition \ref{prop:small}, we however anyways do not expect to obtain optimal powers in our main estimates in Propositions \ref{prop:HT} and \ref{prop:HT2}, we do not optimize in the powers in this lemma.
\end{rmk}

\begin{proof}
In order to simplify notation, for $\gamma \in (0,1)$ and $\delta \in(0,1)$ we introduce
\begin{align*}
J_{\delta,\gamma}:= \{x\in J: \dist(x,\partial J)\geq \delta^{\gamma}\},\ 
J_{\delta,\gamma}^c:= J \setminus J_{\delta,\gamma}.
\end{align*}
With this notation and an application of the fundamental theorem, we have
\begin{equation}
\label{eq:fundth}
\begin{split}
\|\p_2 u\|_{L^2(J_{2\delta,\gamma}\times \{0\})}
&\leq \|\p_2 u\|_{L^2(J_{\delta,\gamma}\times \{\delta\})}
+ \int\limits_{0}^{\delta} \|\p_{22} u\|_{L^2(J_{2\delta,\gamma}\times \{y\})}dy,\\
\|\p_2 u\|_{L^2(J_{2\delta,\gamma}^c\times \{0\})}
&\leq \|\p_2 u\|_{L^2(J_{2\delta,\gamma}\times \{0\})}
+ \int\limits_{0}^{(2\delta)^{\gamma}} \|\p_{12} u\|_{L^2(J\times \{0\})}dx,
\end{split}
\end{equation}
where $\gamma\in(0,1)$ is a constant, which is still to be determined, and $\delta \in(0,1)$ is arbitrary. As $\p_1 \p_2 u = \p_1 g$ in $J\times \{0\}$, we directly infer that
\begin{align}
\label{eq:1der}
\int\limits_{0}^{(2\delta)^{\gamma}} \|\p_{12} u\|_{L^2(J\times \{0\})}dx \leq C \delta^{\gamma}\|\p_1 g\|_{L^2(J)}.
\end{align}
Hence, it suffices to bound $\int\limits_{0}^{\delta} \|\p_{22} u\|_{L^2(J_{2\delta,\gamma}\times \{y\})}dy$, in order to show the claim of the lemma. To this end, we observe that $\p_{11} u = -\p_{22} u$ in $J_{\delta,\gamma}\times (0,\delta]$, which is a consequence of harmonicity and up to the boundary regularity. Therefore,
\begin{align}
\label{eq:kernel}
\p_{22} u
&= - \p_{11} u
 = - \p_{11} (G_N \ast_{x_1} (\eta g)) - \p_{11} (G_N \ast_{x_1} ((1-\eta)g))
 =: f_1 + f_2.
\end{align}
Here $G_N(x) = \ln(|x|)$ is the Neumann Green's function and $\eta$ is a cut-off function which satisfies
\begin{align*}
\eta = 1 \mbox{ in } J_{3 \delta/2,\gamma} \mbox{ and } \eta = 0 \mbox{ in } J_{\delta, \gamma}.
\end{align*}
We estimate the two contributions in \eqref{eq:kernel} separately: Recalling the Fourier representation of the Neumann kernel, which was given in the proof of Lemma \ref{lem:Hilbert}, we observe that
\begin{align}
\label{eq:Fourier}
|\F(f_1)(\xi)|= \left| \xi^2 |\xi|^{-1} \F(\eta g) e^{-|\xi|y} \right|
\leq \left|  |\xi| \F(\eta g)(\xi)  \right|.
\end{align}
Treating $x_2$ as a dummy variable and only considering the integration in $x_1$ (for which we also write $L^2_{x_1}(\R)$ to clarify the relevant variable) therefore leads to
\begin{align*}
\|f_1\|_{L^2_{x_1}(J_{2\delta,\gamma}\times \{y\})}
&\leq \|f_1\|_{L^2_{x_1}(\R \times \{y\})}
\leq  \| \F(\eta g) \|_{\dot{H}^1(\R )}
=  \|\p_1 (\eta g) \|_{L^2(\R)} \\
&\leq C(\|\p_1 g\|_{L^2(J)} + \delta^{-\gamma}\|g\|_{L^2(J)}).
\end{align*}
Thus,
\begin{align}
\label{eq:f1}
\int\limits_{0}^{\delta}\|f_1\|_{L^2(J_{2\delta,\gamma}\times \{y\})}dy
\leq C (\delta \|\p_1 g\|_{L^2(J)} + \delta^{1-\gamma}\|g\|_{L^2(J)}).
\end{align}
For $f_2$ we note that with $x_1 \in J_{2\delta,\gamma}$ and $z\in J^{c}_{3\delta/2, \gamma} $ we have $|\p_{11} G_N(x_1-z,x_2)| \leq C \delta^{-2\gamma} $. Thus,
\begin{align*}
\|f_2\|_{L^2(J_{2\delta,\gamma})}
&\leq \|(\p_{11} G_N)\ast_{x_1} ((1-\eta)g)\|_{L^2(J_{2\delta,\gamma})} 
\leq C\delta^{-2\gamma} \|g\|_{L^1(J^c_{3\delta/2,\gamma})} \\
&\leq C\delta^{-2\gamma}  \delta^{\gamma/2}\|g\|_{L^2(J)}
\leq C\delta^{-3\gamma/2}  \|g\|_{L^2(J)}.
\end{align*}
As a consequence,
\begin{align}
\label{eq:f2}
\int\limits_{0}^{\delta}\|f_2\|_{L^2(J_{2\delta,\gamma}\times \{y\})}dy \leq C\delta^{1-3\gamma/2} \|g\|_{L^2(J)}.
\end{align}
Combining \eqref{eq:f1} and \eqref{eq:f2}, we infer that
\begin{align*}
 \int\limits_{0}^{\delta} \|\p_{22} u\|_{L^2(J_{2\delta,\gamma}\times \{y\})}dy \leq C\left((\delta^{1-3\gamma/2}  + \delta^{1-\gamma})\|g\|_{L^2(J)} + \delta \|\p_1 g\|_{L^2(J)}  \right).
\end{align*}
Together with \eqref{eq:fundth} and \eqref{eq:1der}, this therefore yields
\begin{align}
\label{eq:combi}
\|g\|_{L^2(J)} \leq  C\left((\delta^{1-3\gamma/2}  + \delta^{1-\gamma})\|g\|_{L^2(J)} + (\delta^{\gamma}+\delta) \|\p_1 g\|_{L^2(J)}  \right).
\end{align}
Setting $\gamma = \frac{2}{5}$ and inserting the resulting bounds into \eqref{eq:combi} proves the claimed estimate.
\end{proof}

\begin{rmk}
\label{rmk:replace}
Instead of arguing by relying on the Fourier transform as in \eqref{eq:Fourier}, we could also have used a direct kernel estimate in order to deduce \eqref{eq:f1}. We briefly outline the argument. To this end, let $\psi(x')$ be a smooth cut-off function which equals one on $[-1,1]$ and vanishes outside of $[-2,2]$. Then,
\begin{align*}
f_1 = (\psi \p_1 G) \ast_{x_1} \p_1 (\eta g) + ((1-\psi) \p_1 G) \ast_{x_1} \p_1(\eta g). 
\end{align*}
We estimate these terms by the applying Young's convolution estimate:
\begin{align}
\label{eq:Young}
\begin{split}
\|f_1\|_{L^1_{x_1}(J_{2\delta, \gamma} \times \{y\})}
&\leq \|(\psi \p_1 G) \ast_{x_1} \p_1 (\eta g)\|_{L^1_{x_1}(\R \times \{y\})} + \|((1-\psi) \p_1 G) \ast_{x_1} \p_1 (\eta g)\|_{L^1_{x_1}(\R \times \{y\})}\\
& \leq  \|(\psi \p_1 G)\|_{L^1_{x_1}(\R\times \{y\})}\|  \p_1 (\eta g)\|_{L^2(\R)} \\
& \quad + \|((1-\psi) \p_1 G)\|_{L^2_{x_1}(\R \times \{y\})} \| \p_1 (\eta g)\|_{L^1(\R)}.
\end{split}
\end{align}
For $\delta \in (0,1)$ we integrate in $y\in [0,\delta]$ and estimate the resulting two contributions separately: On the one hand we have
\begin{align}
\label{eq:G1}
\|\psi \p_1 G\|_{L^1(\R\times [0,\delta])}
\leq \left\|\frac{1}{|(x_1,y)|}\right\|_{L^1([-1,1]\times [0,\delta])}
\leq C\delta(1+\log(\delta)).
\end{align}
On the other hand,
\begin{align}
\label{eq:G2}
\|(1-\psi) \p_1 G\|_{L^2(\R \times [0,\delta])}
\leq \left\|\frac{1}{|(x_1,y)|}\right\|_{L^2((\R\setminus [-1,1])\times [0,\delta])}
\leq C \delta.
\end{align}
Combining \eqref{eq:G1}, \eqref{eq:G2} and \eqref{eq:Young} with the compact support of $\eta g$ and Hölder's inequality (to pass from an $L^1$ to an $L^2$ estimate for $\p_1 (\eta g)$) therefore again yields \eqref{eq:f1} (up to a logarithmic loss).
\end{rmk}

\subsection{Applications: Almost invertibility and approximation}
\label{sec:applications}

In this section we exploit the propagation of smallness estimates from the previous section to prove our main results on the truncated Hilbert transform, Propositions \ref{prop:HT} and \ref{prop:HT2}. The section is divided into two parts: In the first part, we deduce the conditional invertibility estimates, in the second, we prove appropriate approximation results.

\subsubsection{Almost invertibility and the proof of Proposition \ref{prop:HT}}
\label{sec:invert}

By combining the estimates from Proposition \ref{prop:small} and Lemma \ref{lem:boundary} from Section \ref{sec:small}, we provide the proof of Proposition \ref{prop:HT}.

\begin{proof}[Proof of Proposition \ref{prop:HT}]
Without loss of generality, we assume that the normalization conditions from Convention \ref{con:int} hold.
With these, the proof of Proposition \ref{prop:HT} is a direct consequence of Proposition \ref{prop:small} and Lemma \ref{lem:boundary}. Indeed, choosing $\epsilon = \delta$ we have the two bounds
\begin{equation}
\label{eq:applications}
\begin{split}
\|\p_2 u\|_{L^2(J \times \{\delta\})}
&\leq  \exp(C(l_I,d_I)(1+\ln(\delta))/\delta^{\sigma}) \|H_J g\|_{L^2(I)} + \frac{\delta}{2}\|g\|_{L^2(J)},\\
\|g\|_{L^2(J )}& \leq C\left( \|\p_2 u\|_{L^2(J \times \{\delta\})} +\delta^{2/5}  \|g\|_{L^2(J )} + \delta^{2/5} \|\p_1 g\|_{L^2(J)})\right).
\end{split}
\end{equation}
Setting
\begin{align*}
\delta : = \min\left\{ \left(\frac{\|g\|_{L^2(J)}}{10(C+1)\|\p_1 g\|_{L^2(J)}}\right)^{5/2}, \frac{1}{100}  \right\},
\end{align*}
and combining the two estimates from \eqref{eq:applications} yields
\begin{align*}
\|\p_1 g\|_{L^2(J)} +\|g\|_{L^2(J)} &\leq  \exp\left( C(d_I, l_I)\left(1  +   \frac{\|\p_1 g\|_{L^2(J)}^{\tilde{\sigma}}}{\|g\|_{L^2(J)}^{\tilde{\sigma}}} \right) \right) \|H_J g\|_{L^2(I)}+ \frac{1}{5}\|g\|_{L^2(J)},
\end{align*}
where $\tilde{\sigma}$ is an arbitrary constant which is strictly larger than $ \sigma$ (in order to absorb the factor $\log(\delta)$).
Absorbing the last term on the right hand side into the left hand side then concludes the proof of Proposition \ref{prop:HT}.
\end{proof}

\begin{rmk}
\label{rmk:comment}
As pointed out in the introduction, compared to the bounds in \cite{APS13}, the result of Proposition \ref{prop:HT} does not have the optimal dependence in the exponential in terms of the power $\tilde{\sigma}$. However, relying only on propagation of smallness estimates, this method of proof is very robust and \emph{does not} require any a priori knowledge of the explicit singular value asymptotics.\\
The estimate \eqref{eq:almost_invert1} can also be read as a \emph{quantitative} unique continuation result for the truncated Hilbert transform. It thus refines the qualitative result from Lemma \ref{cor:inj_dens} (a).
\end{rmk}

\subsubsection{Approximation}

In this section we present the argument for the approximation result from Proposition \ref{prop:HT2}.
It can be viewed as a refinement of the density result from Corollary \ref{cor:inj_dens}, in which we also estimate the \emph{cost} of approximating a given function.\\
In order to construct a suitable control function (within the admissible error threshold for the approximation), we rely on quantitative variational techniques similar as in \cite{FZ00}. To this end,
for each given $h\in H^1(J)$ and $\epsilon>0$, we consider the functional 
\begin{equation}
\label{eq:Ge}
\begin{split}
&\G_{\epsilon}:L^2(J) \rightarrow [-\infty,\infty],\\
&\G_{\epsilon}(g) := \frac{1}{2}\int\limits_{I}|H_J g|^2 dx + \epsilon \|g \|_{L^2(J)} - \int\limits_{J}h g dx.
\end{split}
\end{equation}
In the sequel, we will show that for each function $h\in L^2(J)$ and each error threshold $\epsilon>0$ the functional $\G_{\epsilon}$ has a unique minimizer $\bar{g}$ (c.f. Lemma \ref{lem:ex}). Further, we will then set $f:= -\chi_I H_J \bar{g} \in L^2(I)$ and show that $f$ satisfies the properties claimed in Proposition \ref{prop:HT2} (c.f. Lemma \ref{lem:ex} and the estimates in the following proof of Proposition \ref{prop:HT2}). Again, the quantitative propagation of smallness estimates from Section \ref{sec:small} constitute a key building block in these arguments.\\

We begin by analysing the functional from \eqref{eq:Ge}. The existence of minimizers and their approximation properties rely on coercivity properties of the functional, which are a consequence of the \emph{qualitative} unique continuation properties of the Hilbert transform (c.f. Lemma \ref{cor:inj_dens} (a)). For an estimate on the cost of approximation, we will need more \emph{quantitative} control.\\

Before proving the existence of minimizers to \eqref{eq:Ge}, we discuss a slight simplification of our problem: We observe that without loss of generality we may assume that the function $h$ in Proposition \ref{prop:HT2} is an element of $H^1_0(J)$.
Indeed, if this is not the case, using the notation from Convention \ref{con:int}, we can extend the function $h$ to a function $\tilde{h}$ having compact support in a slightly larger interval $J_{h_1}$ such that 
\begin{align}
\label{eq:compact}
\|\tilde{h}\|_{H^1_0(J_{h_1})}\leq C(d_I) \|h\|_{H^1(J)}, \quad
\|\tilde{h}\|_{L^2(J_{h_1})}\leq C(d_I) \|h\|_{L^2(J)}.
\end{align}
In particular, the fattening factor $h_1$ (which was defined in Convention \ref{con:int}) is chosen such that $J_{h_1}$ and $I$ are still disjoint intervals (c.f. Convention \ref{con:int}), whose distance is comparable to the distance of the original intervals $I,J$. If we can show the approximation property for the two intervals $I,J_{h_1}$, we also infer the approximation property of \eqref{eq:approx} by restricting to the intervals $J,I$. Thus, in the sequel, we will without further comment assume that $h\in H^{1}_0(J)$.\\

With these preliminary considerations, we turn to the existence of minimizers of \eqref{eq:Ge}:

\begin{lem}
\label{lem:ex}
Let $I,J$ be open, bounded intervals such that $\overline{I}\cap \overline{J} = \emptyset$.
Let $\epsilon>0$ and let $\G_{\epsilon}:L^2(J) \rightarrow [-\infty,\infty]$ be as in \eqref{eq:Ge}.
Then there exists a unique minimizer $\bar{g}\in L^2(J)$ of the functional $\G_{\epsilon}$. Moreover, the function $f:=-\chi_I H_J(\bar{g})$ satisfies
\begin{align}
\label{eq:small}
\|H_I f - h\|_{L^2(J)}\leq \epsilon.
\end{align}
\end{lem}

\begin{proof}
We argue in two steps and first prove the existence of minimizers. In a second step, we then deduce the approximation property \eqref{eq:small} associated with $f$. \\

\emph{Step 1: Existence of a minimizer.}
Since for each parameter $\epsilon>0$ and for each function $h\in L^2(J)$ the functional $\G_{\epsilon}$ is  convex and lower semicontinuous with respect to weak convergence on $L^2(J)$, it suffices to prove its coercivity, i.e. the property that
\begin{align*}
\G_{\epsilon}(g) \rightarrow \infty \mbox{ as } \|g\|_{L^2(J)} \rightarrow \infty.
\end{align*}
We claim that this can be reduced to a unique continuation result for the Hilbert transform. Indeed, let $\{g_j\} \subset L^2(J)$ be an (arbitrary) sequence with the property that $\|g_j\|_{L^2(J)} \rightarrow \infty$. Setting $\tilde{g}_j:= \frac{g_j}{\|g_j\|_{L^2(J)}}$, we observe that $\|\tilde{g}_j\|_{L^2(J)}=1$ and that
\begin{align}
\label{eq:normalized}
\frac{\G_{\epsilon}(g_j)}{\|g_j\|_{L^2(J)}} 
= \frac{\|g_j\|_{L^2(J)}}{2} \|H_J \tilde{g}_j\|_{L^2(I)}^2 + \epsilon - \int\limits_{J} h \tilde{g}_j dx.
 \end{align}
 We now distinguish two cases:
 \begin{itemize}
 \item[(a)] If $\liminf\limits_{j\rightarrow \infty} \|H_J \tilde{g}_j\|_{L^2(I)}>0$, equation \eqref{eq:normalized} and the normalization of the functions $\tilde{g}_j$ directly imply that
 \begin{align*}
 \frac{\G_{\epsilon}(g_j)}{\|g_j\|_{L^2(J)}}  \rightarrow \infty \mbox{ as } j \rightarrow \infty.
 \end{align*}
 This then proves the claimed coercivity.
\item[(b)] In the case that
$\liminf\limits_{j\rightarrow \infty} \|H_J \tilde{g}_j\|_{L^2(I)}=0$, we use the normalization of $\tilde{g}_j$ and the characterization of the adjoint Hilbert transform to infer that
\begin{align*}
\tilde{g}_j &\rightharpoonup \tilde{g} \mbox{ in } L^2(J),\\
H_J \tilde{g}_j &\rightharpoonup H_J \tilde{g} \mbox{ in } L^2(I).
\end{align*}
By lower semi-continuity of the $L^2$ norm, $\|H_J \tilde{g}\|_{L^2(I)} = 0$. Lemma \ref{cor:inj_dens} (a) then enforces that $\tilde{g}=0$ as a function in $L^2(J)$. In particular, we infer that
\begin{align*}
\int\limits_J h \tilde{g}_j dx \rightarrow 0 \mbox{ as } j \rightarrow \infty.
\end{align*}
Thus, choosing $j\in \N$ sufficiently large in equation \eqref{eq:normalized}, implies that $ \frac{\G_{\epsilon}(g_j)}{\|g_j\|_{L^2(J)}} \geq \epsilon/2$, which concludes the coercivity proof.
 \end{itemize}
\emph{Step 2: Proof of \eqref{eq:small}.} Let $\bar{g}$ be the minimizer of $\G_{\epsilon}$. Hence, for all $\mu \in \R$ and all $g\in L^2(J)$ it holds, $\G_{\epsilon}(\bar{g})\leq \G_{\epsilon}(\bar{g} + \mu g)$. Spelling this out and applying the triangle inequality yields
\begin{align*}
0&\leq \epsilon(\|\bar{g} + \mu g\|_{L^2(J)} - \|\bar{g}\|_{L^2(J)}) + 
\frac{\mu^2}{2}\|H_J g\|_{L^2(J)}^2  + \mu \left( \int\limits_I H_J \bar{g} H_J g dx - \int\limits_J hg dx \right)\\
&\leq \epsilon \mu \|g\|_{L^2(J)}+ 
\frac{\mu^2}{2}\|H_J g\|_{L^2(J)}^2  + \mu \left( \int\limits_I H_J \bar{g} H_J g dx - \int\limits_J hg dx \right).
\end{align*}
Dividing by $\mu$ and considering the limit $\mu \rightarrow 0_+$ implies
\begin{align*}
0 \leq \epsilon \|g\|_{L^2(J)} + \int\limits_{I} H_J \bar{g} H_J g dx - \int\limits_J hg dx.
\end{align*}
Combining this with the analogous limit $\mu \rightarrow 0_-$ thus result in
\begin{align}
\label{eq:ELG}
\left|\int\limits_{I} H_J \bar{g} H_J g dx - \int\limits_J hg dx \right| \leq  \epsilon \|g\|_{L^2(J)}  \mbox{ for all } g\in L^2(J).
\end{align}
Defining $f:= -\chi_I H_J \bar{g}$ and using that by Lemma \ref{lem:adj} we have that $(f,H_Jg)_{L^2(I)} =(g,H_If)_{L^2(J)}$, therefore entails that
\begin{align*}
\left|\int\limits_{J} (H_I f - h)g dx \right| \leq  \epsilon \|g\|_{L^2(J)}  \mbox{ for all } g\in L^2(J).
\end{align*}
By duality this concludes the proof of \eqref{eq:small}.
\end{proof}

\begin{rmk}
\label{rmk:ELG}
The Euler-Lagrange equations of the variational problem from \eqref{eq:Ge} and the arguments from the preceding proof, which lead to \eqref{eq:ELG}, imply that 
\begin{align*}
\min\limits_{g \in L^2(J)} \G_{\epsilon}(g) = -\frac{1}{2} \|f\|_{L^2(I)}^2.
\end{align*}
\end{rmk}

In order to conclude the proof of the approximation result from Proposition \ref{prop:HT2}, it remains to estimate the \emph{cost} of approximation, i.e. the size of $\|f\|_{L^2(I)}$.
To this end, we exploit improved coercivity properties, which rely on \emph{quantitative} unique continuation results in the form of the propagation of smallness result from Proposition \ref{prop:small}. With these estimates at hand, we argue along the  lines of the controllability proofs in \cite{FZ00}.

\begin{proof}[Proof of Proposition \ref{prop:HT2}]
Without loss of generality, we may assume that the normalization conditions from Convention \ref{con:int} hold.
We consider the functional $\G_{\epsilon}$ and rewrite it as
\begin{align*}
\G_{\epsilon}(g) = \tilde{\G}_{\epsilon,\delta}(g) + \frac{\epsilon}{2} \|g\|_{L^2(J)}
- \int\limits_{J} h (\p_2 v(x,0)-\p_2 v(x,\delta)) dx,
\end{align*}
where 
\begin{align*}
\tilde{\G}_{\epsilon,\delta}(g):= \frac{1}{2} \int\limits_{I} |H_J g|^2 dx - 
\int\limits_{J} h \p_2 v(x,\delta) dx + \frac{\epsilon}{2} \|g\|_{L^2(J)},
\end{align*}
and where $v=N(g)$ (c.f. Lemma \ref{lem:Hilbert}) is a solution to
\begin{align*}
\D v& = 0 \mbox{ in } \R^2_+,\\
\p_2 v & =  \chi_J g \mbox{ on } \R \times \{0\}.
\end{align*}
Here we require that $\delta>0$ is chosen in such a way that for all $g\in L^2(J)$ and for fixed $\epsilon>0$
\begin{align}
\label{eq:delta}
\frac{\epsilon}{2} \|g\|_{L^2(J)} - \int\limits_{J} h (\p_2 v(x,0)- \p_2v(x,\delta)) dx \geq 0 .
\end{align}
Hence,
\begin{align*}
I_1 := \min\limits_{g\in L^2(J)} \G_{\epsilon}(g) \geq \inf\limits_{g\in L^2(J)} \tilde{\G}_{\epsilon,\delta}(g)=:I_2.
\end{align*}
Using the Euler-Lagrange equations for the functional $\G_{\epsilon}$ (c.f. Remark \ref{rmk:ELG}), we infer that
\begin{align*}
I_1 = - \frac{1}{2} \int\limits_{I} |H_J \bar{g}|^2 dx, 
\end{align*}
where $\bar{g}$ denotes the solution to the minimization problem for $\G_{\epsilon}$. Defining $f:= -\chi_I H_J \bar{g}$, thus implies
\begin{align*}
\int\limits_{I}|f|^2 dx \leq -2 I_2.
\end{align*}
Therefore it suffices to estimate $I_2$ and to ensure that \eqref{eq:delta} holds true. Invoking Proposition \ref{prop:small} with $\tilde{\epsilon}:= \epsilon \|h\|_{L^2(J)}^{-1}$ and using Young's inequality leads to 
\begin{align*}
I_2 &= \inf\limits_{g \in L^2(J)} \frac{1}{2}\int\limits_{I}|H_J g|^2 dx 
- \int\limits_{J} h \p_2 v(x,\delta) dx +\frac{\epsilon}{2} \|g\|_{L^2(J)}  \\
& \geq \inf\limits_{g\in L^2(J)}  \frac{1}{2}\int\limits_{I}|H_J g|^2 dx  - \exp(C(d_I, l_I)(1+|\ln(\tilde{\epsilon})|)/\delta^{\sigma})\|H_J g\|_{L^2(I )} \|h\|_{L^2(J)}\\
& \quad  -\frac{\tilde{\epsilon}}{2}\|h\|_{L^2(J)}\|g\|_{L^2(J)} + \frac{\epsilon}{2}\|g\|_{L^2(J)}\\
&\geq - (\exp(C(d_I, l_I)(1+|\ln(\epsilon \|h\|_{L^2(J)}^{-1})|/\delta^{\sigma})+1) \|h\|_{L^2(J)}^2.
\end{align*}
Therefore,
\begin{align}
\label{eq:cost}
\|f\|_{L^2(I)}^2 \leq (\exp(C(d_I, l_I)(1+|\ln(\epsilon \|h\|_{L^2(J)}^{-1})|/\delta^{\sigma})+1)\|h\|_{L^2(J)}^2.
\end{align}
In order to conclude the proof of the approximation result, it thus suffices to derive a relation between $\epsilon$ and $\delta$. This is obtained as a consequence of the requirement \eqref{eq:delta}. To observe this, we note that
\begin{align*}
&\left|\int\limits_{J} h(\p_2 v(x,0) - \p_2 v(x,\delta))|^2 dx \right|
\leq \|h\|_{H^1_0(J)} \|\p_2v(\cdot, 0) - \p_2 v(\cdot, \delta)\|_{\dot{H}^{-1}(J)}\\
& \leq \|h\|_{H^1_0(J)} \int\limits_{0}^{\delta} \|\p_{22}v(\cdot, y) \|_{\dot{H}^{-1}(J)} dy
 = \|h\|_{H^1_0(J)} \int\limits_{0}^{\delta} \|\p_{11}v(\cdot, y) \|_{\dot{H}^{-1}(J)} dy\\
& \leq \|h\|_{H^1_0(J)} \int\limits_{0}^{\delta} \|\p_{1}v(\cdot, y) \|_{L^{2}(J)} dy
 \leq C \delta \|h\|_{H^1_0(J)}\|g\|_{L^2(J)},
\end{align*}
where we used that
\begin{align*}
\|\p_{1}v(\cdot, y) \|_{L^{2}(J)} 
\leq \|\p_{1}v(\cdot, y) \|_{L^{2}(\R^n)} 
= \|\sgn(\xi) e^{-|\xi|y} \F(g)\|_{L^{2}(\R^n)} 
\leq \|g \|_{L^{2}(J)}.
\end{align*}
Combining this estimate with the requirement \eqref{eq:delta} yields the necessary condition
\begin{align*}
\delta \leq \frac{\epsilon}{2C\|h\|_{H^1_0(J)}}.
\end{align*}
Setting $\delta := \frac{\epsilon}{2C\|h\|_{H^1(J)}}$ and inserting this into the bound \eqref{eq:cost} therefore results in
\begin{align*}
\|f\|_{L^2(I)} \leq  \exp(C(d_I,l_I)(1 +\|h\|_{H^1(J)}^{\tilde{\sigma}})/\epsilon^{\tilde{\sigma}})\|h\|_{L^2(I)},
\end{align*}
where $\tilde{\sigma}$ is an arbitrary constant with $\tilde{\sigma}> \sigma$ and where we have used that $h \in H^1_0(J)$ (c.f. the discussion before Lemma \ref{lem:ex}). This implies the estimate for the cost of controllability.
The estimate for the approximation quality of $f$, i.e. the estimate for $\|H_I f-h\|_{L^2(J)}$, is a direct consequence of \eqref{eq:small}.
\end{proof}

\section{Extension to the Higher Dimensional Situation}
\label{sec:nD1}

In this section we extend the invertibility and approximation results for the Hilbert transform to Riesz transforms, which can be viewed as $n$-dimensional analoga of the Hilbert transform. Since the results on and the arguments for these $n$-dimensional operators are analogous to the ones for the Hilbert transform, we only present short sketches of the proofs and point out their key ingredients.\\

We recall that in $\R^{n}$ the $j$-th Riesz transform $R_j$ is defined by $\F(R_j f)(\xi) = \frac{\xi_j}{|\xi|} \F (f)(\xi)$. As in the case of the Hilbert transform, we note that it is possible to realize this operator by considering an associated harmonic extension into $\R^{n+1}_+$:

\begin{lem}
\label{lem:Riesz_trafo}
Let $f\in L^2(\R^n)$ and denote its Neumann harmonic extension by $N(f):= G_N \ast_{x'} f$, where $x=(x', x_{n+1})$ and $G_N(x)=|x|^{1-n}$ denotes the Neumann Green's function in $\R^{n+1}_+$. Then $R_j f(x') = \p_j N(f)(x', 0)$.
\end{lem}

Similarly as in the case of the Hilbert transform, we consider the truncated Riesz transform with respect to open, bounded Lipschitz sets:

\begin{defi}
\label{defi:truncated_Riesz}
Let $\Omega_1, \Omega_2 \subset \R^n$ be open, bounded Lipschitz sets. Assume that $f\in L^2(\Omega)$. Then we define the \emph{$j$th truncated Riesz transform} with respect to $\Omega_1, \Omega_2$ as
\begin{align*}
R_j^{\Omega_1,\Omega_2} f := \chi_{\Omega_2} R_j \chi_{\Omega_1} f.
\end{align*}
If there is no danger of confusion, we also abbreviate $R_j^{\Omega_1, \Omega_2}$ by $R_j^{\Omega_1}$. Moreover, we define $R^{\Omega_1, \Omega_2}$ as the vector $(R_{1}^{\Omega_1, \Omega_2},\dots, R_{n}^{\Omega_1, \Omega_2})$.
\end{defi}

We collect several properties of the truncated Riesz transforms:

\begin{lem}
\label{lem:prop_Riesz}
Let $\Omega_1, \Omega_2 \subset \R^n$ be open, bounded Lipschitz sets with $\overline{\Omega}_1 \cap \overline{\Omega}_2 = \emptyset$. Then the following properties hold:
\begin{itemize}
\item[(a)] The operators $R_j^{\Omega_1, \Omega_2}$ with $j\in\{1,\dots,n\}$ are smoothing operators.
\item[(b)] The Hilbert space adjoint of $R_j^{\Omega_1, \Omega_2}$ is given by $R_j^{ \Omega_2, \Omega_1}$. 
\item[(c)] If for some $f\in L^2(\Omega_1)$ and for all $j\in \{1,\dots,n\}$ it holds that $R_j^{\Omega_1, \Omega_2} f (x) = 0$ in $\Omega_2$, then $f=0$ as a function in $L^2(\Omega_1)$.
\item[(d)] The set
\begin{align*}
\mathcal{S}:= \{\chi_{\Omega_2} R_j f: f \in C^{\infty}_0(\Omega_1) \mbox{ and } j\in\{1,\dots,n\}\}
\end{align*}
is dense in $L^2(\Omega_2)$.
\end{itemize}
\end{lem}

\begin{proof}
Property (a) is a consequence of the kernel representation of $R_j$; property (b) follows either from Plancherel and the multiplier definition of $R_j$, or from integration by parts in combination with the extension point of view. In order to infer (c), we note that $R_j^{\Omega_1, \Omega_2} f = 0$ in $\Omega_2$ for all $j\in \{1,\dots,n\}$ implies that $f= c$ in $\Omega_2$. The claim then follows from using (weak) boundary unique continuation for the associated harmonic extension (as in the analogous Lemma \ref{cor:inj_dens}).\\
Last but not least the density result is a consequence of the Hahn-Banach theorem and properties (b) and (c) from above: Indeed, if $v\in L^2(\Omega_2)$ has the property that
\begin{align*}
(R_j f, v)_{L^2(\Omega_2)} = 0 \mbox{ for all } f\in C^{\infty}_0(\Omega_1) \mbox{ and } j\in\{1,\dots,n\},
\end{align*}
property (b) from above yields that
\begin{align*}
(f, R_j (\chi_{\Omega_2} v))_{L^2(\Omega_1)} = 0 \mbox{ for all } f\in C^{\infty}_0(\Omega_1) \mbox{ and } j\in\{1,\dots,n\}.
\end{align*}
Thus, for all $j\in \{1,\dots,n\}$ it holds that $R_j(\chi_{\Omega_2}v) =0$ on $\Omega_1$. But then property (c) implies that $v=0$, which proves the desired density result.
\end{proof}

As in the case of the Hilbert transform, the \emph{qualitative} injectivity and density properties from Lemma \ref{lem:prop_Riesz} (c) and (d) can be refined. The \emph{quantitative} counterpart of the result from Lemma \ref{lem:prop_Riesz} is given by the following proposition.

\begin{prop}[Quantitative unique continuation]
\label{prop:quant_unique2_Riesz}
Let $\Omega_1,\Omega_2$ be open, bounded Lipschitz sets such that $\overline{\Omega}_1\cap \overline{\Omega}_2 = \emptyset$.
Denote by $R^{\Omega_2}:L^2(\R^n) \rightarrow L^2(\R^n)$ the truncated Riesz transform with respect to $\Omega_2$ and assume that 
 $g\in H^1(\Omega_2)$.
Then, there exist constants $\tilde{\sigma}>0$, $C>1$ depending only on the geometry of $\Omega_1, \Omega_2$ and their relative locations in $\R^{n}$ such that
 \begin{align}
 \label{eq:almost_invert_Riesz}
  \| g\|_{H^1(\Omega_2)}
  \leq  \exp\left( C\left( 1+\frac{\|g\|_{H^1(\Omega_2)}^{\tilde{\sigma}} }{\| g\|_{L^2(\Omega_2)}^{\tilde{\sigma}}} \right) \right) \|R^{\Omega_2} g\|_{L^2(\Omega_1)}.
 \end{align} 
\end{prop}

\begin{proof}
Since the arguments for this result are analogous to the ones for the Hilbert transform, we only point out the main ingredients. Similarly as in the setting of the Hilbert transform, they consist two parts:
\begin{itemize}
\item As the first an main ingredient, we rely on an (interior) propagation of smallness result. As in Proposition \ref{prop:small} we have that for any $\delta \in (0,1)$ and any $\epsilon>0$
\begin{equation}
\label{eq:quant_control_1}
\begin{split}
\|\p_{n+1} u\|_{L^2(\Omega_2 \times \{\delta\})} 
&\leq \exp(C(\Omega_1,\Omega_2)(|\ln(\epsilon)|+2)/\delta^{\sigma}) \|R^{\Omega_2} g\|_{L^2(\Omega_1)} + \frac{\epsilon}{2} \|g\|_{L^2(\Omega_2)}.
\end{split}
\end{equation}
The proof of this result proceeds analogously as the one for Proposition \ref{prop:small}, at which point we had not made substantial use of the one-dimensionality of the set-up. As main steps we use a combination of an interior propagation of smallness result and the Lebeau and Robbiano bulk-boundary interpolation estimate \cite{LR95}.
\item As the second ingredient in the proof of \eqref{eq:almost_invert_Riesz} we relate the norms $\|\p_{n+1} u\|_{L^2(\Omega_2 \times \{\delta\})} $ and $\|\p_{n+1} u\|_{L^2(\Omega_2 \times \{0\})} $, which follows from analogous arguments as in Lemma \ref{lem:boundary} (fundamental theorem, $L^2$ based Green's function estimates by direct Fourier methods).
\end{itemize}
As in the argument in Section \ref{sec:invert} this then yields the desired result.
\end{proof}

As an analogue of Proposition \ref{prop:HT2}, and as a refinement of Lemma \ref{lem:prop_Riesz} (d), we also present a \emph{quantitative} approximation result for the truncated Riesz transforms.

\begin{prop}
\label{prop:approx_Riesz}
Let $\epsilon>0$ and let $\Omega_1,\Omega_2 \subset \R^n$ be open, bounded Lipschitz sets with $\overline{\Omega}_1\cap \overline{\Omega}_2 = \emptyset$.
Denote by $R^{\Omega_1}_j:L^2(\R^n) \rightarrow L^2(\R^n)$ the $j$th truncated Riesz transform with respect to $\Omega_1$ and assume that $h\in H^1(\Omega_2)$.
Then there exist constants $\sigma>0$, $C>1$ (which only depend on the geometries of $\Omega_1, \Omega_2$ and their relative locations) and functions $f_1,\dots, f_n\in L^2(\Omega_1)$ such that
\begin{align}
\label{eq:approx_Riesz}
\|h- \sum\limits_{j=1}^{n} R_j^{\Omega_1} f_j\|_{L^2(\Omega_2)} \leq \epsilon  \mbox{ and }
\|f\|_{L^2(\Omega_1)} \leq e^{C(1+\|h\|_{H^1(\Omega_2)}^{\sigma}/\epsilon^{\sigma})} \|h\|_{L^2(\Omega_2)},
\end{align}
where $f=(f_1,\dots,f_n)$.
\end{prop}

\begin{proof}
The proof is similar as the one from Proposition \ref{prop:HT2}. The only additional difficulty stems from the fact that the density result of Lemma \ref{lem:prop_Riesz} (d) and the quantitative propagation of smallness estimate \eqref{eq:quant_control_1} requires information on \emph{all} partial Riesz transforms. In order to take this into account, we consider a slightly modified functional:
\begin{align*}
\Je: L^2(\Omega_2) \rightarrow [-\infty,\infty], \ 
\Je(g) := \frac{1}{2}\|R^{\Omega_2} g \|_{L^2(\Omega_1)}^2 + \epsilon \|g\|_{L^2(\Omega_1)} - (h,g)_{L^2(\Omega_2)}.
\end{align*}
Here we use the Euclidean norm to define the $L^2$ norm of the vector valued function $R^{\Omega_2}g$. The existence of minimizers to this functional can again be reduced to proving coercivity for the functional, which in turn only involves the qualitative unique continuation result from Lemma \ref{lem:Riesz_trafo} (c). Defining $f_j:= - R_j^{\Omega_2} \bar{g}$, where $\bar{g}$ denotes the minimizer of the functional $\Je$ and computing the Euler-Lagrange equations for the functional as in the proof of Lemma \ref{lem:ex} thus yields the first bound from \eqref{eq:approx_Riesz}.\\
In order to deduce the estimate on the cost of controllability, we argue as in the proof of Proposition \ref{prop:HT2} and consider the auxiliary functional
\begin{align*}
&\Jed: L^2(\Omega_2) \rightarrow [-\infty, \infty], \\
& \Jed(g) = \frac{1}{2}\|R^{\Omega_2} g\|^2_{L^2(\Omega_1)} - (h,\p_{n+1} N(g)(x',\delta))_{L^2(\Omega_2)} + \frac{\epsilon}{2}\|g\|_{L^2(\Omega_2)},
\end{align*}
where $\delta>0$ is a suitable constant, which is still to be chosen. Rewriting 
\begin{align*}
\Je(g) = \Jed(g) + \frac{\epsilon}{2}\|g\|_{L^2(\Omega_2)} - (h,\p_{n+1} N(g)(x',0) - \p_{n+1} N(g)(x',\delta) )_{L^2(\Omega_2)},
\end{align*}
and choosing $\delta>0$ subject to the condition that for all $g\in L^2(\Omega_2)$
\begin{align*}
\frac{\epsilon}{2}\|g\|_{L^2(\Omega_2)}- (h,\p_{n+1} N(g)(x',0) - \p_{n+1} N(g)(x',\delta) )_{L^2(\Omega_2)} \geq 0,
\end{align*}
implies that it suffices to estimate $I_2:=\inf\limits_{g\in L^2(\Omega_2)} \Jed(g)$. Replacing Proposition \ref{prop:small}
by \eqref{eq:quant_control_1} allows us to use exactly the same ideas as in the derivation of Proposition \ref{prop:HT2}, we hence omit the details of this.
\end{proof}

\begin{rmk}
\label{rmk:choice_of_norms}
There exist several alternative formulations of the approximation result from Proposition \ref{prop:approx_Riesz}. Instead of considering sums of Riesz transforms and a vector $f=(f_1,\dots,f_n)$, it would for instance also have been possible to consider a single -- in a sense the maximal -- Riesz transform. In this case the approximation in the first estimate in \eqref{eq:approx_Riesz} would have worked with a single function $f$. Here the approximation result would have turned into the following statement: There are constants $\sigma>0$, $C>1$ (depending only on the geometries of $\Omega_1, \Omega_2$ and the relative location of these two sets) such that for each $\epsilon>0$ there exist a function $f\in L^2(\Omega_1)$ and a value $j\in\{1,\dots,n\}$ such that
\begin{align}
\label{eq:approx_Riesz1}
\|h-  R_j^{\Omega_1} f\|_{L^2(\Omega_2)} \leq \epsilon  \mbox{ and }
\|f\|_{L^2(\Omega_1)} \leq e^{C(1+\|h\|_{H^1(\Omega_2)}^{\sigma}/\epsilon^{\sigma})} \|h\|_{L^2(\Omega_2)}.
\end{align}
To deduce this result, we can for instance consider the functional
\begin{align*}
\tilde{\Je}(g) = \frac{1}{2}\max\limits_{j\in \{1,\dots,n\}}\|R_j^{\Omega_2} g\|_{L^2(\Omega_1)}^2  + \epsilon \|g\|_{L^2(\Omega_1)} - (h,g)_{L^2(\Omega_2)}
\end{align*}
and argue along the lines of the proof of Proposition \ref{prop:HT}.
\end{rmk}

\section{Perturbations of Riesz Transforms}
\label{sec:Perturb}

It is possible to embed the previous two examples into a slightly more general class of operators, which are given as $L^2(\R^n)$ adjoints of ``Riesz type transforms" of more general elliptic operators. To specify this, we consider a uniformly elliptic operator $L =  \p_i a^{ij}(x')\p_j $, $i,j\in\{1,\dots,n\}$,  with smooth, uniformly elliptic, symmetric coefficients $a^{ij}$. Its associated ``harmonic" extension (c.f. \cite{CS07}, \cite{ST10}, \cite{CS16}) is defined as
\begin{align}
\label{eq:var_L}
L u + \p_{n+1}^2 u & = 0 \mbox{ in } \R^{n+1}_+,\\
\p_{n+1} u & = f \mbox{ on } \R^{n} \times \{0\}.
\end{align}
For simplicity and due to the obvious distinction in the behaviour of the Neumann Green's functions in two and higher dimensions, we only consider the situation $n>1$ (in the case $n=1$ similar arguments would work in Hardy and BMO spaces; instead of presenting the details for this, we however refer to \cite{KN85}, \cite{DM95}): For $f\in C^{\infty}_0(\R^n)$ we consider energy solutions of this, i.e. we define a function $u:\R^{n+1}_+ \rightarrow \R$ to be a solution, if it is a distributional solution in the sense that
\begin{align*}
(a^{ij}\p_i u, \p_j \varphi)_{L^2(\R^{n+1}_+)} +(\p_{n+1} u, \p_{n+1} \varphi)_{L^2(\R^{n+1}_+)} = (f,\varphi)_{L^2(\R^n \times \{0\})}\\
 \mbox{ for all } \varphi \in C^{\infty}_0(\overline{\R}^{n+1}_+),
\end{align*} 
and if $u\in \dot{H}^{1}(\R^{n+1}_+)\cap L^{2^*}(\R^{n+1}_+)$, where $2^*:= \frac{2n}{n-2}$ denotes the Sobolev embedding exponent. By energy and trace estimates and Sobolev embedding, we obtain that for any $x_{n+1}\geq 0$
\begin{align*}
\|u\|_{L^q(\R^n \times \{x_{n+1}\})}
&\leq C\|u\|_{\dot{H}^{1/2}(\R^n \times \{x_{n+1}\})}
\leq C\|u\|_{\dot{H}^{1}(\R^{n+1}_+)}\\
&\leq C \|f\|_{\dot{H}^{-1/2}(\R^n \times \{0\})}
\leq C \|f\|_{L^p(\R^n \times \{0\})},
\end{align*}
where $p=\frac{2n}{1+n}$ and $q=p' = \frac{2n}{n-1}$ (and $n>1$). Thus, by the Schwartz kernel theorem, there exists a kernel $K_{x_{n+1}}(x',y')$ with the property that $u$ can be represented by it:
\begin{align*}
u(x',x_{n+1}) = \int\limits_{\R^n \times \{0\}}K_{x_{n+1}}(x',y') f(y') dy'.
\end{align*}
By virtue of duality and the symmetry of $a^{ij}$ it follows that $K_{x_{n+1}}(x',y') = K_{x_{n+1}}(y',x')$. 

We remark that using this representation formula, it is for instance possible to deduce quantitative estimates of $K_{x_{n+1}}(x',y')$ in the tangential directions. To this end, we fix $x_{n+1}=1$.
Considering a compactly supported function $f$, a fixed point $y' \in \R^{n}$ and a further point $x=(x',1)$ with $\dist(x',\supp(f))=2$ implies that in $B_1(x)$ the function $u$ and its derivatives are bounded. Thus, for all $\alpha = (\alpha',0) \in \N^{n+1}$ the mapping $f \mapsto \p_x^{\alpha} u(x)$ is a linear continuous map, which has a kernel representation. Repeating this argument also for the second variable (by duality and Sobolev embedding) hence implies that for the points $x,y$ as above
\begin{align*}
|\p_x^{\alpha} K_{1}(x',y')|\leq C(\alpha) \mbox{ for all } \alpha \in \N^n.
\end{align*}
By (tangential) rescaling, this in particular entails that the kernel $K_{1}(x',y')$ and its first order tangential derivatives enjoy the same bounds as the Neumann Green's function for the Laplacian. 
Estimates for the $x_{n+1}$ dependence follow from scaling.

In analogy to the notation in the previous sections we set $G^{a}_N(x,y):=K_{x_{n+1}}(x',y')$.\\

For this definition of a solution of the extension problem, i.e. $u:= G^{a}_N\ast_{x'} f$, and for $j\in\{1,\dots,n\}$ we study the mapping
\begin{equation}
\begin{split}
 \label{eq:Gen_Hilbert}
 R_j^a: L^2(\R^n) \rightarrow L^2(\R^n),\
 f \mapsto \p_j u (x',0),
\end{split}
\end{equation}
where the existence of and the bounds for the corresponding non-tangential limits can for instance be obtained through a Rellich-Necas-Payne-Weinberger formula \cite{McLean} and suitable approximation arguments.
In analogy to the previous section, we call the resulting operators $R_j^{a}$ the \emph{$j$th generalized 
Riesz transforms associated with the operator $L$}. Similarly as before, we define the \emph{truncated generalized Riesz transforms} as
\begin{align*}
R_j^{a,\Omega_1,\Omega_2} := \chi_{\Omega_2} R_j^a \chi_{\Omega_1}. 
\end{align*}
Moreover, we let $A^a_j$ denote the corresponding adjoint operators, i.e. we assume that
\begin{align*}
(A^a_j (\chi_{\Omega_2} f),g)_{L^{2}(\Omega_1)} = (f, R_j^{a,\Omega_1,\Omega_2} g)_{L^2(\Omega_2)} \mbox{ for all } f\in L^2(\Omega_2).
\end{align*}
Unlike previously, it is non-trivial to express these $L^2$ adjoints explicitly. As, however, all the approximation properties from the previous section only depended on properties of the adjoint, i.e. of the uniformly elliptic operator $L+\p_{n+1}^2$, our main results remain valid in the present, more general set-up. \\

In particular, we infer the following conditional invertibility and \emph{quantitative} unique continuation result:

\begin{prop}[Quantitative unique continuation]
\label{prop:quant_unique2_Riesz_a}
Let $L$ be as above.
Let $\Omega_1,\Omega_2$ be open, bounded Lipschitz sets such that $\overline{\Omega}_1\cap \overline{\Omega}_2 = \emptyset$.
Denote by $R^{a,\Omega_2}:L^2(\R^n) \rightarrow L^2(\R^n)$ the truncated Riesz transform with respect to $\Omega_2$ and $L$, and assume that 
 $g\in H^1(\Omega_2)$.
 Then, there exists constants $\tilde{\sigma}>0$, $C>1$ (depending only on the geometries of $\Omega_1, \Omega_2$ and their relative locations in $\R^n$) such that
 \begin{align}
 \label{eq:almost_invert_Riesz_1}
  \| g\|_{H^1(\Omega_2)}
  \leq  \exp\left( C\left(1 +\frac{\|g\|_{H^1(\Omega_2)}^{\tilde{\sigma}} }{\| g\|_{L^2(\Omega_2)}^{\tilde{\sigma}}} \right) \right) \|R^{a,\Omega_2} g\|_{L^2(\Omega_1)}.
 \end{align} 
\end{prop}

\begin{proof}
This follows as in the constant coefficient case by noting that both central ingredients, the propagation of smallness estimate from Proposition \ref{prop:small} and the regularity results from Lemma \ref{lem:boundary} remain valid. Indeed, Proposition \ref{prop:small} only used three balls and the boundary-bulk
interpolation argument from \cite{LR95}, which are both still true for uniformly elliptic equations with sufficiently smooth coefficients. Lemma \ref{lem:boundary} relied on the fundamental theorem and non-tangential limits, which are true in much rougher settings \cite{K96}.
\end{proof}

Similarly, it is also possible to deduce quantitative approximation results:

\begin{prop}
\label{prop:approx_Riesz_a}
Let $L$ be as above.
Let $\epsilon>0$ and let $\Omega_1,\Omega_2 \subset \R^n$ be open, bounded Lipschitz sets with $\overline{\Omega}_1\cap \overline{\Omega}_2 = \emptyset$.
Denote by $R^{a, \Omega_2}_j:L^2(\R^n) \rightarrow L^2(\R^n)$ the $j$th truncated Riesz transform with respect to $\Omega_1$ and let $A_j^{a}$ denote its $L^2$ adjoint.
Further assume that $h\in L^2(\Omega_2)$.
Then there exist a universal constant $\sigma>0$ and functions $f_1,\dots, f_n\in L^2(\Omega_1)$ such that
\begin{align}
\label{eq:approx_Riesz_1}
\|h- \sum\limits_{j=1}^{n} A_j^a f_j\|_{L^2(\Omega_2)} \leq \epsilon  \mbox{ and }
\|f\|_{L^2(\Omega_1)} \leq e^{C(1+\|h\|_{H^1(\Omega_2)}^{\sigma}/\epsilon^{\sigma})} \|h\|_{L^2(\Omega_2)},
\end{align}
where $f=(f_1,\dots,f_n)$ and where $C = C(\Omega_1,\Omega_2)>1$ is a constant, which only depends on the geometries of $\Omega_1, \Omega_2$ and their relative locations in $\R^n$.
\end{prop}

Although we cannot rely on an explicit expression for $A_j^{a}$, the good stability properties of its adjoint yield sufficient information to infer this quantitative approximation result. On a technical level this is reflected in the fact that the proof of Proposition \ref{prop:approx_Riesz_a} only indirectly uses $A_j^{\Omega_1}$ and mainly exploits properties of $R_j^{a,\Omega_2}$. 

\begin{proof}
There are nearly no changes with respect to the constant coefficient setting, if one replaces the functional $\mathcal{J}_{\epsilon}$ by 
\begin{align*}
\mathcal{J}_{a,\epsilon}(g):= \frac{1}{2}\|R^{a,\Omega_2}g\|_{L^2(\Omega_1)}^2 + \epsilon \|g\|_{L^2(\Omega_1)} -(h,g)_{L^2(\Omega_2)}.
\end{align*}
\end{proof}

\begin{rmk}[Generalizations]
\label{rmk:lower_order}
As a further generalization of the situation discussed above, it is possible to consider operators with rougher coefficients, e.g. $C^{1,1}$ regularity for $a^{ij}$ would have sufficed. Moreover, we can also include first order terms in the operator $L$. However, zeroth order terms destroy even the qualitative unique continuation properties, e.g. the analogue of Lemma \ref{lem:Riesz_trafo} (c) (as for instance it is no longer possible to subtract constants without modifying the equations). 
\end{rmk}

\bibliographystyle{alpha}
\bibliography{citationsHT}

\end{document}